\documentclass[11pt,a4paper]{article}
\usepackage[english]{babel}
\usepackage{amsfonts,a4wide,url}
\usepackage{amsmath}
\usepackage{amssymb}
\usepackage{amscd}
\newtheorem{theorem}{Theorem}
\newtheorem{lemma}{Lemma}
\newtheorem{remark}{Remark}
\newtheorem{example}{Example}
\newtheorem{corollary}{Corollary}
\newtheorem{theoremKM}{Theorem KM}

\newenvironment{proof}{\noindent\textit{Proof}.}{\hfill\raisebox{-1ex}{$\boxtimes$}}
\newenvironment{proofthree}{\noindent\textit{Proof of Theorem~\ref{t:09}.}}{\hfill\raisebox{-1ex}{$\boxtimes$}}
\newenvironment{proofscf}{\noindent\textit{Proof of Theorem~\ref{t:05}}.}{\hfill\raisebox{-1ex}{$\boxtimes$}}
\newcounter{tmpabcd}

\newcommand{\N}{\mathbb{N}}
\newcommand{\Z}{\mathbb{Z}}
\newcommand{\Q}{\mathbb{Q}}
\newcommand{\R}{\mathbb{R}}
\newcommand{\cC}{\mathcal{C}}
\newcommand{\cS}{\mathcal{S}}
\newcommand{\cR}{\mathcal{R}}
\newcommand{\cF}{\mathcal{F}}
\newcommand{\ccF}{\overline{\mathcal{F}}}
\newcommand{\hg}{\hat{\vv g}}
\newcommand{\ve}{\varepsilon}
\newcommand{\eqdef}{:=}
\newcommand{\intr}{\operatorname{int}}
\newcommand{\diag}{\ensuremath{{\rm diag}}}
\newcommand{\SD}{\tilde{\nabla}}
\newcommand{\vv}[1]{{\mathbf{#1}}}
\newcommand{\Sprindzuk}{Sprind\v zuk}

\begin{document}

\title{\huge Explicit bounds for rational points near planar curves\\ and metric Diophantine approximation}

\author{\Large Victor Beresnevich\footnote{EPSRC Advanced Research Fellow, grant no.EP/C54076X/1\hfill\today} \and\Large Evgeniy Zorin}

\date{}

\maketitle

\vspace*{-2ex}

\begin{abstract}
\normalsize
The primary goal of this paper is to complete the theory of metric Diophantine approximation initially developed in \cite{Beresnevich-Dickinson-Velani-07:MR2373145} for $C^3$ non-degenerate planar curves. With this goal in mind, here 
for the first time we obtain fully explicit bounds for the number of rational points near planar curves. Further, introducing a perturbational approach we bring the smoothness condition imposed on the curves down to $C^1$ (lowest possible). This way we 
broaden the notion of non-degeneracy in a natural direction and introduce a new topologically complete class of planar curves to the theory of Diophantine approximation.
In summary, our findings improve and complete the main theorems of \cite{Beresnevich-Dickinson-Velani-07:MR2373145} 
and extend the celebrated theorem of Kleinbock and Margulis \cite{Kleinbock-Margulis-98:MR1652916}
in dimension 2 beyond the notion of non-degeneracy.
\end{abstract}



\large

\section{Introduction}




Problems about rational points lying near curves and surfaces are widespread in number theory and include, for instance, questions regarding small values of homogeneous polynomials on the integer lattice. Within this paper we study the distribution of rational points near curves $\cC$ embedded in $\R^2$. With this in mind we now introduce some basic notation. First of all, without loss of generality, let us agree that the curves we consider are given as a graph $\cC_f=\{(x,f(x)):x\in I\}$ of some function $f$ defined on an interval $I$. Given $\delta>0$, $Q>1$ and a subinterval $J\subset I$, consider the following counting function
$$
N_f(Q,\delta,J)\ \eqdef\ \# \left\{ \left(p_1/q,p_2/q\right)\in\Q^2 \ :
\begin{array}{l}
 p_1/q\in J,\ \ 0<q\le Q, \\[0.5ex]
 |f(p_1/q)-p_2/q|\le\delta Q^{-1}
\end{array}
  \right\},
$$
where $\#A$ stands for the cardinality of a set $A$ and $q,p_1,p_2$ denote coprime integers. Essentially, this function counts rational points $(p_1/q,p_2/q)$ with denominator $q\le Q$ lying at the distance comparable to $\delta/Q$ from the arc $\{(x,f(x)):x\in J\}$ of $\cC_f$.

To begin with, we give a brief account of known results. Let $I\subset\R$ be a compact interval, $c_2\ge c_1>0$ and let $\cF(I;c_1,c_2)$ be the set of $C^2$ functions $f:I\to\R$ such that
\begin{equation} \label{cond_g2p}
c_1 \leq |f''(x)| \leq c_2\qquad\text{for all $x \in I$.}
\end{equation}
In 1994 Huxley \cite[Th.\,1]{Huxley-1994-rational_points} proved that
\begin{equation}\label{Hux}
N_f(Q,\delta,I)\ll_{\ve} C^{10/3}\delta^{1-\ve}Q^2+C^{1/3}Q,\qquad\text{where }\ C=\max\{c_2,c_1^{-1}\}
\end{equation}
and the constant implicit in $\ll_\ve$ depends on $\ve$ but does not depend on $\delta$, $Q$, $C$ or $f$.
For $\delta>Q^{2/3}$ Huxley's result was improved by Vaughan and Velani \cite[Th.\,1]{Vaughan-Velani-2007} on showing that for any $f\in\cF(I;c_1,c_2)$, any $Q>1$ and $0<\delta<\frac12$
\begin{equation}\label{VV1}
N_f(Q,\delta,I)\ll \delta Q^2+\delta^{-\frac12}Q.
\end{equation}
Additionally assuming that $f''\in\operatorname{Lip}_\theta(I)$ with $0<\theta<1$ they proved \cite[Th.\,3]{Vaughan-Velani-2007} that for any $\ve>0$
\begin{equation}\label{VV2}
N_f(Q,\delta,I)\ll \delta Q^2+\delta^{-\frac12}Q^{\frac12+\ve}+\delta^{\frac{\theta-1}2}Q^{\frac{3-\theta}2}.
\end{equation}
Estimates (\ref{VV1}) and (\ref{VV2}) were also extended to an inhomogeneous case in \cite{Beresnevich-Vaughan-Velani-08-Inhom}.

When $\delta=o(Q^{-1})$ the quantity  $N_f(Q,\delta,I)$ can vary from $0$ to $Q$ depending on the choice of $f\in\cF(I;c_1,c_2)$ irrespectively of the actual value of $\delta$  -- see \cite[\S2.2]{Beresnevich-SDA1} for examples. In the case $\delta\gg Q^{-1}$ lower bounds for $N_f(Q,\delta,I)$ were obtained in \cite[Th.\,6]{Beresnevich-Dickinson-Velani-07:MR2373145} and in the inhomogeneous form in \cite[Th.\,5]{Beresnevich-Vaughan-Velani-08-Inhom}. More precisely, it was shown that for any $f\in\cF(I;c_1,c_2)$ additionally satisfying the condition $f\in C^3(I)$ there exist constants $k_1,k_2,c,Q_0>0$ such that for any $Q>Q_0$ and any $\delta$ satisfying $k_1Q^{-1}\le\delta\le k_2$ one has
\begin{equation}\label{BDVV}
N_f(Q,\delta,I)\ge c\delta Q^2.
\end{equation}
However, the above undoubtedly remarkable results fall short of providing a complete theory for the whole class $\cF(I;c_1,c_2)$. Indeed, within (\ref{VV2}) and (\ref{BDVV}) the additional differentiability and Lipschitz assumptions are imposed while (\ref{Hux}) is not optimal. In this paper we resolve this issue in full in relation to lower bounds and furthermore expand the results to a genuinely larger, topologically closed class of functions $f$ introduced in the next paragraph. Furthermore, our results are fully explicit and uniform in $f$.

Throughout $\ccF(I;c_1,c_2)$ stands for the closure of $\cF(I;c_1,c_2)$ in the $C^0$ (uniform convergence) topology, \emph{i.e.} the topology induced by the norm $\|f\|_{C^0}\eqdef\sup_{x\in I}|f(x)|$. In order to state the main counting result of this paper we now gather the definitions of various explicit constants appearing in the course of establishing this result and depending on $c_1$ and $c_2$ only. These are
\begin{equation}\label{mainconst}
    \hat E\ \eqdef \ \frac{3^62^5c_2}{c_1\min\{1,\sqrt{c_1}\}}\,,\qquad c_0 \ \eqdef\ 2^{-13}\hat E^{-6}c_2^{-1},
\end{equation}
\begin{equation}\label{def_const}
   C_1 \ \eqdef\ \frac{c_2}{2c_1 c_0^2}=\frac{2^{25}\hat E^{12}c_2^{3}}{c_1}\,,\qquad C_2\eqdef \left(\frac{c_2C_1}{2c_0}\right)^{1/3}.
\end{equation}

\begin{theorem}\label{t:01}
Let $I$ be a compact interval, $c_2\ge c_1>0$ and the constants $\hat E$, $c_0$, $C_1$ and $C_2$ be given by \eqref{mainconst} and \eqref{def_const}. Let a subinterval $J\subset I$ of length $|J|\le\tfrac12$, $Q>1$ and $\delta\le1$ satisfy the conditions
\begin{equation}\label{e:003}
    \delta Q^2|J|\ge 8C_1
\qquad\text{and}\qquad
    Q\delta \ge C_2.
\end{equation}
Assume also that either
\begin{equation}\label{vb21}
Q\ge \frac{128}{c_0c_1^2}\,|J|^{-3}\qquad\text{or}\qquad
\frac{c_0^2}{c_2}\delta^{-2}\ge Q\ge \frac{16}{c_0c_1^{2}}\,|J|^{-2}.
\end{equation}
Then for any $f\in\ccF(I;c_1,c_2)$
$$N_f(Q,\delta,J)\ge \frac{1}{4C_1}\,\delta Q^2|J|.$$
\end{theorem}

Theorem~\ref{t:01} is a consequence of a more general covering result (Theorem~\ref{t:02} below) which will require the following notation. Let
$$
  \cR^c_f(Q,\delta,J) \ \eqdef \ \left\{ (q,p_1,p_2)\in \N\times\Z^2 \ :
\begin{array}{l}
 p_1/q\in J,\ c Q< q\le Q \\[0.2ex]
 |f(p_1/q)-p_2/q|\le\delta/Q\\[0.2ex]
 \gcd(q,p_1,p_2)=1
\end{array}
  \right\},
$$
where $Q>1$, $\delta>0$, $c\ge0$ and $J\subset I$; and let
$$
\Delta^{c}_f(Q,\delta,J,\rho)\eqdef\bigcup_{(q,p_1,p_2)\in\cR^{c}_f(Q,\delta,J)}
\big\{x:|x-p_1/q|\le \rho\big\}.
$$
Also let $|A|$ denote the Lebesgue measure of a measurable set $A\subset\R^n$.

\begin{theorem}\label{t:02}
Let $I$ be a compact interval, $c_2\ge c_1>0$ and the constants $\hat E$, $c_0$, $C_1$ and $C_2$ be given by \eqref{mainconst} and \eqref{def_const}. Let a subinterval $J\subset I$ of length $|J|\le\tfrac12$, $Q>1$ and $\delta\le1$ satisfy \eqref{e:003}
and \eqref{vb21}. Then for any $f\in\ccF(I;c_1,c_2)$
\begin{equation}\label{vb+x}
\left|\Delta^{c_0}_f(Q,\delta,J,\rho)\cap J\,\right| \ \ge \ \tfrac12 \, |J|,\qquad\text{where $\rho\eqdef C_1(\delta Q^2)^{-1}$.}
\end{equation}
\end{theorem}

The proof of Theorem~\ref{t:01} modulo Theorem~\ref{t:02} is easy and left to the reader, but see \cite[\S4.1]{Beresnevich-Dickinson-Velani-07:MR2373145} for a hint.

Since the constant in (\ref{Hux}) implied by the Vinogradov symbol is independent of $f$, this estimate can also be extended to the class $ \ccF(I;c_1,c_2)$. We state this formally as

\begin{theorem}\label{t:03}
Let $I$ be a compact interval and $c_2\ge c_1>0$. Then \eqref{Hux} remains true for any $f\in\ccF(I;c_1,c_2)$, where
the implicit constant does not depend on $\delta$, $Q$, $C$ or $f$.
\end{theorem}

We believe that (\ref{VV2}) can also be extended to $f\in\ccF(I;c_1,c_2)$, however this requires techniques of a very different nature and we plan to return to this issue in a subsequent publication.

\bigskip
\bigskip
\bigskip

For the rest of this section we discuss various consequences of the above results to metric Diophantine approximation. In what follows unless otherwise mentioned we follow the terminology of Bernik and Dodson \cite{BernikDodson-1999} and Kleinbock and Margulis \cite{Kleinbock-Margulis-98:MR1652916}. The foundations of a general metric theory of Diophantine approximation for planar curves was laid by Schmidt \cite{Schmidt-1964a} in 1964 who proved that every non-degenerate planar curve is extremal (in the sense of \Sprindzuk{} \cite{Sprindzuk-1980-Achievements}). Recall that a curve $\cC_f$ defined as a graph of a $C^2$ function $f:I\to\R$  is non-degenerate if $f''(x)\not=0$ almost everywhere (for the definition of non-degeneracy in higher dimensions see \cite{Beresnevich-02:MR1905790} or \cite{Kleinbock-Margulis-98:MR1652916}). In this case we will also say that $f$ is non-degenerate. In particular, by definition, any $f\in\cF(I;c_1,c_2)$ is non-degenerate for any choice of $c_2\ge c_1>0$. In the case of approximation by linear forms Baker \cite{Baker-1978} refined Schmidt's theorem with a Hausdorff dimension result and recently Badziahin \cite{Badziahin-10} established the inhomogeneous version of Baker's theorem. Furthermore, non-degenerate curves have been shown to be of Groshev type \cite{Beresnevich-Bernik-Dickinson-Dodson-99:MR1807055, Bernik-Dickinson-Dodson-98:MR1632820} (see also \cite{Beresnevich-02:MR1905790, Beresnevich-Bernik-Kleinbock-Margulis-02:MR1944505, Bernik-Kleinbock-Margulis-01:MR1829381} for higher dimensional results). Unlike the dual case, the progress with simultaneous approximation was rather slow. For quite a while nothing was known apart from Bernik's Khintchine type theorem for convergence for parabola \cite{Bernik-1979}. However, in the last 5 years or so a general theory of simultaneous Diophantine approximation was developed in \cite{Beresnevich-Dickinson-Velani-07:MR2373145, Vaughan-Velani-2007}, which was subsequently generalised to multiplicative Diophantine approximation \cite{Badziahin-Levesley-07:MR2347267, Beresnevich-Velani-07:MR2285737} and to the inhomogeneous case \cite{Beresnevich-Vaughan-Velani-08-Inhom}. In short, the progress was based on the development of the theory of ubiquitous systems \cite{Beresnevich-Dickinson-Velani-06:MR2184760} and on the study of the distribution of rational points near planar curves. In particular, the various results on metric Diophantine approximation on planar curves inherited the extra smoothness and/or Lipschitz conditions imposed within (\ref{VV2}) and (\ref{BDVV}). Theorem~\ref{t:02} enables us to remove these indeed unnecessary constrains within the divergence results and furthermore broaden them to a genuinely larger class, which is now introduced.

The curve $\cC_f=\{(x,f(x)):x\in I\}$ (resp. the function $f$) will be called \emph{weakly non-degenerate at $x_0\in I$} if there exist constants $c_2\ge c_1>0$ and a compact subinterval $J\subset I$ centred at $x_0\in J$ such that $f|_J\in\ccF(J;c_1,c_2)$. We will say that $\cC_f$ (resp. the function $f$) is \emph{weakly non-degenerate} if $\cC_f$ is weakly non-degenerate at almost every point $x\in I$. Clearly every non-degenerate curve $\cC_f$ is weakly non-degenerate. However the converse is not always true as follows from the example of \S\ref{what_they_are}, which shows that a weakly non-degenerate curve may be degenerate everywhere.

Theorem~\ref{t:04} below gathers the various main consequences of Theorems~\ref{t:02} and \ref{t:03} for the simultaneous Diophantine approximation on weakly non-degenerate planar curves. Before stating the result we introduce some further notation. Given an arithmetic function $\psi:\N\to(0,+\infty)$, let
$$
\lambda_\psi\eqdef\liminf\limits_{h\to\infty}\dfrac{- \log \psi(h)}{\log h}
$$
denote \emph{the lower order of $1/\psi$ at infinity}. Also define the following two sets of $\psi$-approximable points:
$$
\cS_f(\psi)\eqdef\big\{x\in I:\max\{\|qx\|,\|qf(x)\|\}<\psi(q)
\quad\text{ for i.m. }q\in\N\Big\},
$$
and
$$
\cS^*_f(\psi)\eqdef\big\{x\in I:\|qx\|\cdot\|qf(x)\|<\psi(q)
\quad\text{ for i.m. }q\in\N\Big\},
$$
where $\|y\|=\min\{|y-p|:p\in\Z\}$ and `i.m.' stands for `infinitely many'.

\begin{theorem}\label{t:04}
Let $\psi:\N\to(0,+\infty)$ be monotonic and $f:I\to\R$ be a weakly non-degenerate function. Then
\begin{itemize}
\item[{\rm(A)}]
$\cS_f(\psi)$ has full Lebesgue measure in $I$ whenever $\sum_{h=1}^\infty\psi(h)^2=\infty;$
\item[{\rm(B)}]
$\mathcal{H}^s(\cS_f(\psi)) =  \infty$ whenever
$\sum_{h =1}^{ \infty}  h^{1-s}  \psi(h)^{s+1}   = \infty$ and $s\in(\frac12,1);$
\item[{\rm(C)}] $\dim\cS_f(\psi)=s_0\eqdef\frac{2-\lambda_\psi}{1+\lambda_\psi}$ whenever $\lambda_\psi \in [1/2, 1)$ and $f$ is weakly non-degenerate everywhere apart from a set of Hausdorff dimension $\le s_0$.
\item[{\rm(D)}] $\dim\cS^*_f(\psi)=s^*_0\eqdef\frac{2}{1+\lambda_\psi}$ whenever $\lambda_\psi>1$ and $f$ is weakly non-degenerate everywhere apart from a set of Hausdorff dimension $\le s^*_0$. In particular, $\cC_f$ is strongly extremal.
\end{itemize}
\end{theorem}

The proofs of parts (A), (B) and (C) of Theorem~\ref{t:04} are essentially the same as those of Theorems~1, 3 and 4 in \cite{Beresnevich-Dickinson-Velani-07:MR2373145} with the only differences being that we use our Theorem~\ref{t:02} instead of \cite[Theorem~7]{Beresnevich-Dickinson-Velani-07:MR2373145} and we use Theorem~\ref{t:03} instead of Huxley's original result (\ref{Hux}).
Note also that the proofs make use of continuous differentiability of $f$ - a property that will be shown in the next section (Theorem~\ref{t:06}).
The proof of part (D) of Theorem~\ref{t:04} follows the line of argument of Theorems~6 and $6^*$ from \cite{Beresnevich-Velani-07:MR2285737}. For the modifications are obvious we leave further details out. Using our Theorem~\ref{t:02} in combination with the ideas of \cite{Beresnevich-Vaughan-Velani-08-Inhom} it is also straightforward to state and prove an inhomogeneous version of Theorem~\ref{t:04}.

\medskip

Weakly non-degenerate curves are characterised by the property that locally they can be perturbed into an arbitrarily close `properly' non-degenerate curve with `rigid' bounds on their curvature.
By these we mean that the constants $c_1$ and $c_2$ appearing in (\ref{cond_g2p}) are not varying as we perturb the curve.
Considering how Diophantine properties of manifolds are affected by small perturbations is not absolutely new. For example, Rynne \cite{Rynne-03:MR1950434} obtained a negative result by showing that certain Diophantine properties of non-degenerate manifolds are \textsc{not} preserved under small perturbations even in the $C^k$ topology. It is likely that establishing positive results will require imposing some kind of rigidity on the geometry of perturbed manifolds, likewise conditions (\ref{cond_g2p}) hold uniformly within $\cF(I;c_1,c_2)$. This gives rise to the following

\medskip

\noindent\textbf{General problem.} Find a `reasonable' generalisation of weak non-degeneracy for manifolds in higher dimensions and prove that such manifolds are (strongly) extremal and/or satisfy the analogues of the Khintchine-Groshev theorem (see \cite{Beresnevich-02:MR1905790, Beresnevich-Bernik-Kleinbock-Margulis-02:MR1944505, Beresnevich-Dickinson-Velani-07:MR2373145, BernikDodson-1999, Bernik-Kleinbock-Margulis-01:MR1829381, Kleinbock-Margulis-98:MR1652916} for appropriate terminology and related results).

\section{Functions in $\ccF(I;c_1,c_2)$\/: what they are}\label{what_they_are}

In this section we make an attempt to understand the main object we study in this paper -- functions in $\ccF(I;c_1,c_2)$. In particular, it is mandatory to understand whether this class is any bigger than $\cF(I;c_1,c_2)$.

Since the functions $f\in\ccF(I;c_1,c_2)$ are obtained as limits of continuous and even twice differentiable functions in the uniform convergence topology, they are continuous. As is well known, differentiability is not preserved by the limit functions in the $C^0$ topology; for example, any continuous function on a compact interval can be uniformly approximated by a polynomial (Weierstrass' theorem). However, we shall see that function in $\ccF(I;c_1,c_2)$ are indeed continuously differentiable. Furthermore, they happen to have the second derivatives almost everywhere. On the other hand, we shall see that the second derivative may be non-existent on an everywhere dense set and so may be discontinuous everywhere. The latter fact in particular, shows that the class $\ccF(I;c_1,c_2)$ is genuinely bigger than $\cF(I;c_1,c_2)$ and thus the notion of weak non-degeneracy is not vacuous. Throughout this section, $I=[x_1,x_2]$ is a compact interval and $c_2\ge c_1>0$.

We begin by investigating the convexity properties of functions in $\cF(I;c_1,c_2)$. The function $f:I\to\R$ will be called \emph{$(c_1,c_2)$-convex}\/ if for every $x\in I$ and every $\delta>0$ such that $x\pm\delta\in I$
\begin{equation} \label{cond_g2finite}
    c_1 \leq \frac{f(x+\delta)-2f(x)+f(x-\delta)}{\delta^2} \leq c_2\,.
\end{equation}
In what follows $\cC(I;c_1,c_2)$ will be the set of $(c_1,c_2)$-convex functions.

\emph{Historical note.} In the case $c_2=+\infty$ the r.h.s. of (\ref{cond_g2finite}) imposes no restriction on $f$. Consequently, $f$ is called \emph{$c_1$-convex} or simply \emph{strongly convex}. The strongly convex functions (also known as \emph{uniformly convex functions}) were introduced by Levitin and Poljak \cite{Levitin-Poljak-66:MR0211590}, and have been widely used over the past 50 years mostly in optimization and mathematical finance. Geometrically, the function $f$ is $c_1$-convex if for every $x$ in the interior of $I$ the radius $R$ of the supporting circle of the graph $\cC_f$ is bounded above by $c_1^{-1}$ (see also \cite{Vial-82:MR637263, Vial-83:MR707055} for other properties of strongly convex functions). In the case $c_2<+\infty$ the r.h.s. inequality of (\ref{cond_g2finite}) implies that $R$ is also bounded below by $c_2^{-1}$. In the case $f\in C^2(I)$ we have that
$f\in\cC(I;c_1,c_2)$ if and only if $c_1\le f''(x)\le c_2$ for all $x\in I$. In fact, the following theorem shows that $\cC(I;c_1,c_2)$ coincides with the topological closure of the set of functions satisfying the latter condition.

\begin{theorem}\label{t:05}
$\ccF(I;c_1,c_2)=\cC_\pm(I;c_1,c_2)\eqdef \{\pm f:f\in\cC(I;c_1,c_2)\}$.
\end{theorem}

We will use well known properties of convolution.
Given $\phi,\psi:\R\rightarrow\R$, the \emph{convolution} of $\phi$ and $\psi$ is the function $(\phi\star\psi):\R\to\R$ defined by
$$
(\phi\star\psi)(x)\eqdef\int_{-\infty}^{+\infty}\phi(x-t)\psi(t)d t.
$$
There are various assumption ensuring its existence. We will use the following well known

\begin{lemma}\label{convolution_Cinfty}
If $\phi \in C^{\infty}(\R)$, $\int_{-\infty}^{+\infty}|\phi(x)|d x<\infty$ and $\psi:\R\rightarrow\R$ is bounded and integrable, then $\phi\star\psi\in C^{\infty}(\R)$.
\end{lemma}

\begin{proofscf}
As noted above $\cF(I;c_1,c_2)\subset\cC_\pm(I;c_1,c_2)$. In fact, the latter is an easy consequence of Taylor's formula. Clearly,  taking the limit $f\to f_0$, where $f\in\cF(I;c_1,c_2)$ and $f_0\in\ccF(I;c_1,c_2)$, preserves (\ref{cond_g2finite}), thus showing the inclusion $\ccF(I;c_1,c_2)\subset\cC_\pm(I;c_1,c_2)$. The main substance of the proof is therefore to establish that $\cC_\pm(I;c_1,c_2)\subset\ccF(I;c_1,c_2)$. Let $f\in\cC(I;c_1,c_2)$. Define $\hat{f}:\R\rightarrow\R$ by setting
$$
    \hat{f}(x)\eqdef\left\{
    \begin{array}{ccl}
    f(x) & \text{ if } & x\in I\eqdef[x_1,x_2],\\
    f(x_1) &  \text{ if } & x<x_1 \\
    f(x_2) & \text{ if } & x>x_2.
    \end{array}\right.
$$
Clearly, $\hat f$ is uniformly bounded and continuous on $\R$ and identically equal to $f$ on $I$. Further, since $I$ is compact and $\hat f$ is constant outside $I$, it is easily seen that $\hat f$ is uniformly continuous on $\R$. Define $B:\R\rightarrow\R$ by setting
\begin{equation}\label{def_B}
    B(x)\eqdef\left\{
    \begin{array}{ccl}
      \displaystyle \exp\Big(-\tfrac{1}{(x-1)^2}-\tfrac{1}{(x+1)^2}\Big) & & \text{if}\quad  |x|< 1, \\[1.5ex]
      0 & & \text{otherwise}.
    \end{array}
\right.
\end{equation}
It is easily verified that $B\in C^\infty(\R)$ and is supported on $[-1,1]$. Then
\begin{equation}\label{def_w}
    w\eqdef\int_{-\infty}^{+\infty}B(x)d x=\int_{-1}^{1}B(x)d x.
\end{equation}
Given an $\ve>0$, define $ f_{\ve}:\R\rightarrow\R$ by
\begin{equation}\label{def_ge}
     f_{\ve}(x)\eqdef\frac{1}{w\ve}\int_{-\infty}^{+\infty}B\Big(\frac{x-y}{\ve}\Big) \hat{f}(y)d y.
\end{equation}
The function $ f_{\ve}(x)$ is $1/(w\ve)$ times the convolution of $\hat{f}(x)$ and $B(\frac{x}{\ve})$. By Lemma~\ref{convolution_Cinfty}, $ f_{\ve}(x)\in C^{\infty}(\R)$.
Making the change of variables $z=\frac{x-y}{\ve}$ transforms (\ref{def_ge}) into
\begin{equation}\label{def_ge2}
     f_{\ve}(x)=\frac{1}{w}\int_{-\infty}^{+\infty}B(z)\hat{f}(x-\ve z)d z=\frac{1}{w}\int_{-1}^{1}B(z)\hat{f}(x-\ve z)d z.
\end{equation}
By the uniform continuity of $\hat f$, for any $\eta>0$ there is an $\ve>0$ such that
 \begin{equation}\label{vb22}
    \sup_{|x'-x|\le \ve}|\hat f(x')-\hat f(x)|<\eta.
\end{equation}
Then
$$
\begin{array}{rcl}
\displaystyle\sup_{x\in\R}| f_{\ve}(x)-\hat{f}(x)| & \stackrel{\eqref{def_w},\,\eqref{def_ge2}}{=} &
\displaystyle\sup_{x\in\R}\left|\frac1w\int_{-1}^1B(z)\big(\hat f(x-\ve z)-\hat{f}(x)\big)d z\right|\\[3ex]
&\le &
\displaystyle\frac1w\int_{-1}^1B(z)d z\sup_{|x'-x|\le \ve}\big|\hat f(x')-\hat{f}(x)\big|\ \stackrel{\eqref{def_w},\,\eqref{vb22}}{<}\eta.
\end{array}
$$
This means that $ f_{\ve}$ converges to $\hat f$ uniformly on $\R$ as $\ve\to0$.

Since $ f_{\ve}\in C^\infty(\R)$, using Taylor's formula we verify that  for all $x\in \R$
\begin{equation}\label{vb23}
    \lim_{\delta\rightarrow 0}\frac{ f_{\ve}(x+\delta)-2 f_{\ve}(x)+ f_{\ve}(x-\delta)}{\delta^2}= f_{\ve}''(x).
\end{equation}
By (\ref{def_ge2}),
\begin{equation}\label{ge2p}
    \begin{array}[b]{lc}
    \displaystyle\frac{ f_{\ve}(x+\delta)-2 f_{\ve}(x)+ f_{\ve}(x+\delta)}{\delta^2}\\[2ex] \displaystyle \qquad=
    \frac{1}{w}\int_{-1}^{1}B(z)\frac{\hat{f}(x-\delta z+\delta)-2\hat{f}(x-\delta z)+\hat{f}(x-\delta z-\delta)}{\delta^2}d z.
    \end{array}
\end{equation}

When $x\in[x_1+2\delta,x_2-2\delta]$, where $[x_1,x_2]=I$, we have that $x-\delta z\pm\delta\in I$ for any $z\in[-1,1]$. Then, since $\hat f\in\cC(I;c_1,c_2)$, the fraction within the r.h.s. of (\ref{ge2p}) is bounded between $c_1$ and $c_2$. Consequently, by (\ref{def_w}), the l.h.s. of (\ref{ge2p}) is bounded between $c_1$ and $c_2$ for all $x\in[x_1+\delta,x_2-\delta]$. By (\ref{vb23}),
\begin{equation} \label{cond_g2p_v2}
    c_1\leq f_{\ve}''(x)\leq c_2
\end{equation}
for all $x\in[x_1+2\delta,x_2-2\delta]$. Since $\delta$ can be made arbitrarily small and $ f_{\ve}''(x)$ is continuous on $I$, (\ref{cond_g2p_v2}) must hold on $I$. This means that $ f_{\ve}\in\cF(I;c_1,c_2)$ and consequently $f$ belongs to $\ccF(I;c_1,c_2)$ as a uniform limit of $f_\ve$. In the case $-f\in\cC(I;c_1,c_2)$ taking $- f_{\ve}$ does the job and completes the proof.
\end{proofscf}

\medskip

We now utilise the characterisation of functions in $\ccF(I;c_1,c_2)$ given by Theorem~\ref{t:05} to show that these functions are actually continuously differentiable.

\begin{theorem} \label{t:06}
$\cC(I;c_1,c_2)\subset C^1(I)$.
\end{theorem}
\begin{proof}
Let $f\in\cC(I;c_1,c_2)$. As a convex function $f$ has left and right derivatives $f_+'(x)$ and $f_-'(x)$ at each point $x$ of the interior of $I$ (to be denoted as $\intr I$) -- see e.g. \cite[Theorem~1.3.3]{Nicolescu-Persson-06:ISBN9780387243009}. Furthermore, for all $x,y\in\intr I$ such that $x<y$
\begin{equation}\label{ld_leq_rd}
    f_-'(x)\leq f_+'(x)\leq f_-'(y)\leq f_+'(y).
\end{equation}
By definition, $f_-'(x)=f_+'(x)$ if and only if $f$ is differentiable at $x$. Assume for the moment that $f$ is not differentiable at some point $x_0\in\intr I$, that is, by (\ref{ld_leq_rd}),
\begin{equation}\label{fin_diff_lr}
    f_-'(x_0)<f_+'(x_0).
\end{equation}
Define the auxiliary function $t:I\rightarrow\R$ by setting
\begin{equation}\label{def_t}
    t(x)\eqdef\left\{\begin{array}{ccl}
    f_-'(x) & \text{if} & x<x_0\\
    f_+'(x) & \text{if} & x\geq x_0
    \end{array}\right.
\end{equation}
This function is known as a \emph{subdifferential} for $f$ (see \cite[\S5]{Nicolescu-Persson-06:ISBN9780387243009} for its definition and basic properties). By Theorem~1.6.1 of~\cite{Nicolescu-Persson-06:ISBN9780387243009}, for any $a,b\in I$ we have that
\begin{equation}\label{intPropertySubDifferential}
    f(b)-f(a)=\int_a^b t(x)d x.
\end{equation}
The latter inequality implies that
$$
\begin{array}{ccl}
    \displaystyle\frac{f(x_0+\delta)+f(x_0-\delta)-2f(x_0)}{\delta^2} & = & \displaystyle\frac{\int_{x_0}^{x_0+\delta} f_+'(x)d x-\int_{x_0-\delta}^{x_0} f_-'(x)d x}{\delta^2} \\[2ex] &\stackrel{\eqref{ld_leq_rd}}{\leq}&  \displaystyle\frac{f_+'(x_0)\delta-f_-'(x_0)\delta}{\delta^2}\\[2ex]
    &=&\displaystyle\frac{f_+'(x_0)-f_-'(x_0)}{\delta}.
\end{array}
$$
By (\ref{fin_diff_lr}), the latter fraction tends to infinity as $\delta\to0$. This contradicts to the fact that $f\in\cC(I;c_1,c_2)$. The contradiction shows that $f$ is everywhere differentiable.

To complete the proof we have to verify that $f'$ is continuous. Property~(\ref{ld_leq_rd}) implies that the function $f'$ is monotonically increasing on $I$. So, by the Monotone Convergence Theorem, for any $x_0\in\intr I$ there exist left and right limits of $f'$ and furthermore we have that
\begin{equation}\label{leftLimit_rightLimit}
    \lim_{x\rightarrow x_0^-}f'(x)\leq\lim_{x\rightarrow x_0^+}f'(x).
\end{equation}
By definition,
\begin{equation} \label{CisC1:step1}
    f_-'(x_0)=\lim_{\delta\rightarrow 0+}\frac{f(x_0)-f(x_0-\delta)}{\delta}.
\end{equation}
By Theorem~1.6.1 of~\cite{Nicolescu-Persson-06:ISBN9780387243009}, for $\delta>0$ we have
\begin{equation} \label{CisC1:step2}
    f(x_0)-f(x_0-\delta)=\int_{x_0-\delta}^{x_0}f'(x)d x
\end{equation}
and, by the monotonicity of $f'$,
\begin{equation} \label{CisC1:step3}
   \int_{x_0-\delta}^{x_0}f'(x)d x\leq\delta\lim_{x\rightarrow x_0^-}f'(x).
\end{equation}
Combining (\ref{CisC1:step1}), (\ref{CisC1:step2}) and~(\ref{CisC1:step3}) shows that
\begin{equation} \label{CisC1:step4}
    f_-'(x_0)\leq\lim_{x\rightarrow x_0^-}f'(x).
\end{equation}
Similarly, we establish that
\begin{equation} \label{CisC1:step5}
    \lim_{x\rightarrow x_0^+}f'(x)\le f_+'(x_0).
\end{equation}
By the differentiability of $f$, $f_-'(x_0)=f_+'(x_0)$. Thus~(\ref{CisC1:step4}), (\ref{CisC1:step5}) together with~(\ref{leftLimit_rightLimit}) imply that
$\lim_{x\rightarrow x_0^-}f'(x)=\lim_{x\rightarrow x_0^+}f'(x)$ and complete the proof.
\end{proof}

\medskip

Since functions in $\ccF(I;c_1,c_2)$ are continuously differentiable, we are able to simplify property (\ref{cond_g2finite}) and thus give an alternative description of the class $\ccF(I;c_1,c_2)$. With this goal in mind we now introduce further notation. Let $\cC^1(I;c_1,c_2)$ be the set of $C^1(I)$ functions such that
\begin{equation} \label{ie_derivative}
    c_1\leq\frac{f'(x+\delta)-f'(x)}{\delta}\leq c_2
\end{equation}
for any $x\in I$ and $\delta>0$ with $x+\delta\in I$. Also let $\cC_{\pm}^1(I;c_1,c_2)=\{\pm f:f\in\cC_{\pm}^1(I;c_1,c_2)\}$.

\begin{theorem}\label{t:07}
$\ccF(I;c_1,c_2)=\cC^1_\pm(I;c_1,c_2)$.
\end{theorem}

\begin{proof}
Using the Mean Value Theorem, it is easily seen that $\cF(I;c_1,c_2)\subset \cC_{\pm}^1(I;c_1,c_2)$. By Theorem~\ref{t:06}, we are able to take the limit (in the uniform convergence topology) within (\ref{ie_derivative}) as $f\to f_0$, where $f\in\cF(I;c_1,c_2)$ and $f_0\in\ccF(I;c_1,c_2)$. Thus we trivially have that $\ccF(I;c_1,c_2)\subset \cC_\pm^1(I;c_1,c_2)$. To show the opposite inclusion we use functions $ f_{\ve}$ similarly to the proof of Theorem~\ref{t:06}. The details are easy and are left as an exercise.
\end{proof}

\begin{remark}\rm
The obvious consequence of Theorem~\ref{t:07}, or indeed (\ref{ie_derivative}), is that the functions in $\ccF(I;c_1,c_2)$ have bi-Lipschitz derivatives. However, the following example shows that their second derivative may be non-existent on an everywhere dense set.
\end{remark}

\begin{example} \label{ex_cf}\rm
Let $A\subset (x_1,x_2)$ be a countable set (e.g, $A=(x_1,x_2)\cap\Q$), where $[x_1,x_2]=I$. We will construct a function $f$ which fails to have the second derivative exactly on $A$. Since $A$ is countable, we can fix a bijection $\phi:A\rightarrow\N$. Let $(c_n)$ be a sequence of positive numbers such that $\sum_{n=1}^\infty c_n<\infty$, \emph{e.g.} $c_n=n^{-2}$. Let $T\eqdef\sum_{n=1}^{\infty}c_n$ and
\begin{equation} \label{ex_cf:def_t}
    t(x)\eqdef1+\sum_{{a\in A,\, a<x}}c_{\phi(a)}.
\end{equation}
The function $t(x)$ is well defined because the sum $\sum c_n$ is absolutely convergent. It is easy to verify that $t(x)$ is strictly increasing and positive on $I$. Moreover, $t(x)$ is continuous at any point of $I\setminus A$ and discontinuous at any point of $A$ -- see e.g. \cite[p.18]{Gelbaum-Olmsted-03:MR1996162}. Since $T=\sum c_n$,  $t(x)$ is bounded above by $T+1$ for any $x\in I$.

\noindent Now define $v:I\rightarrow\R$ by setting
$$
    v(x)\eqdef \int_{x_1}^xt(z)dz\qquad\text{for }x\in I.
$$
The function $v$ is well defined as $t$ is continuous almost everywhere. Also $v$ is strictly increasing because $t(x)\ge1$ for all $x\in I$. Also, as $t$ is bounded on $I$, $v$ is continuous at every point of $I$, and as $t$ is continuous at every point $A\setminus I$, $v$ is differentiable at every point of this set. On the contrary, if $a\in A$, by the definition of $t$, one readily computes that $v(a+\delta)-v(a-\delta)\ge c_{\phi(a)}$ for any $\delta>0$ and therefore $v$ is not differentiable at $a$.
Finally, let
$$
    f(x)\eqdef\int_{x_1}^xv(z)dz\qquad\text{for }x\in I.
$$
As $v$ is continuous, $f$ is continuously differentiable and $f'=v$. However, by what we have seen above, $f$ fails to have the second derivative on $A$. Our final goal is to verify that $f$ satisfies (\ref{ie_derivative}). Given the definition of $f$ and $v$, (\ref{ie_derivative}) transforms into
$$
    c_1\leq\frac{1}{\delta}\int_{x}^{x+\delta}t(z)\leq c_2
$$
when $x\in I$, $\delta>0$ and $x+\delta\in I$. The latter inequalities are satisfied with $c_1=1$ and $c_2=T+1$ because $1\le t(x)\le 1+T$ for all $x\in I$. Thus, $f\in\cC^1_\pm(I;c_1,c_2)=\ccF(I;c_1,c_2)$.
\end{example}

\begin{remark}\rm
By Alexandrov's theorem \cite[Theorem 3.11.2]{Nicolescu-Persson-06:ISBN9780387243009}, any convex function has the second derivative almost everywhere. Thus, the functions in $\ccF(I;c_1,c_2)$ are almost everywhere twice differentiable. It is easy to deduce from the fact that $\cC(I;c_1,c_2)=\ccF(I;c_1,c_2)$ that if the second derivative of $f\in\ccF(I;c_1,c_2)$ exists at some point $x_0$, it necessary satisfies the inequalities $c_1\leq |f''(x_0)|\leq c_2$. Although, $f''$ exists at every point except a set $A$ of Lebesgue measure 0, this exceptional set $A$ can be everywhere dense and so $f''$ may be discontinues everywhere on $I$. Note also that the above example can be modified to show that the set of points where the second derivative does not exist is an uncountable set of Hausdorff dimension 1.
\end{remark}

\section{Reduction to $C^2$ functions}

The goal of this section is to show that it is sufficient to prove Theorem~\ref{t:02} for $\cF(I;c_1,c_2)$ only. This follows from the following 

\begin{theorem}\label{t:08}
Let $\cF$ be a set of continuous functions on an interval $I$ and $\ccF$ be the closure of $\cF$ in the uniform convergence topology. Let positive numbers $c,Q,\delta,\rho$ and a subinterval $J\subset I$ be fixed. Then
\begin{equation}\label{vb0}
\inf_{\bar f\in\ccF}\left|\Delta^{c}_{\bar f}(Q,\delta,J,\rho)\cap J\,\right| \ = \ \, \inf_{f\in\cF}\left|\Delta^{c}_f(Q,\delta,J,\rho)\cap J\,\right| .
\end{equation}
\end{theorem}

\begin{proof} Since $\cF\subset\ccF$, the l.h.s. of (\ref{vb0}) is less than or equal to the r.h.s of (\ref{vb0}). By the definition of $\Delta^{\delta}_f(Q,\delta,J,\rho)$, to complete the proof it suffices to verify that for any $\bar f\in\ccF$ there exists $f\in\cF$ such that $\cR^c_{f}(Q,\delta,J)\subset\cR^c_{\bar f}(Q,\delta,J)$.

Let $\cR^*$ denote the set of coprime integer triples $(q,p_1,p_2)$ such that $cQ\le q\le Q$, $p_1/q\in J$ and $|\bar f(p_1/q)-p_2/q|\le 1+\delta/Q$. It is easy to see that $\cR^*$ is finite and strictly larger than $\cR^c_{\bar f}(Q,\delta,J)$. Therefore
$$
\ve^*\eqdef\min\{|\bar f(p_1/q)-p_2/q|-\delta/Q\ :\ (q,p_1,p_2)\in\cR^*\setminus \cR^c_{\bar f}(Q,\delta,J)\}
$$
is positive and well defined. Let $\ve=\min\{1,\ve^*\}$. Take any function $f\in\cF$ such that $\sup_{x\in I}|f(x)-\bar f(x)|<\ve$. Since $\ccF$ is the closure of $\cF$ in the $C^0(I)$ topology, such a function $f$ exists. Using the definitions of $\ve$ and $f$ one readily verifies that $\cR^c_{f}(Q,\delta,J)\subset\cR^*$. Assume for the moment that there is a $(q,p_1,p_2)\in \cR^c_{f}(Q,\delta,J)\setminus\cR^c_{\bar f}(Q,\delta,J)$. Then, by the definition of $\ve^*$, we have that
\begin{equation}\label{eq1}
 |\bar f(p_1/q)-p_2/q|\ge\delta/Q+\ve^*.
\end{equation}
On the other hand, since $(q,p_1,p_2)\in \cR^c_{f}(Q,\delta,J)$, we have that
$$
|\bar f(p_1/q)-p_2/q|\le |\bar f(p_1/q)-f(p_1/q)|+|f(p_1/q)-p_2/q|< \ve+\delta/Q\le \ve^*+\delta/Q.
$$
This contradicts to (\ref{eq1}), thus showing the inclusion $\cR^c_{f}(Q,\delta,J)\subset\cR^c_{\bar f}(Q,\delta,J)$ and completing the proof.
\end{proof}

\bigskip

Similar arguments are applied to Theorem~\ref{t:03}. The details are easy and left to reader.

\section{An explicit version of Theorem BKM}\label{effBKM}

Theorem~BKM mentioned in the above heading is Theorem~1.4 from \cite{Bernik-Kleinbock-Margulis-01:MR1829381} due to Bernik, Kleinbock and Margulis. We will be interested in the case $n=2$. The goals of this section are (i) to generalise it to weakly non-degenerate maps; and (ii) to make it effective and indeed fully explicit. Our approach to achieving these goals develops the ideas of \cite{Beresnevich-SDA1} and \cite{Bernik-Kleinbock-Margulis-01:MR1829381}.

\subsection{Statement of results}

We will be interested in maps $\vv g=(g_1,g_2): I \rightarrow \R^2$ given by
\begin{equation} \label{def_f}
\begin{array}{l}
g_1(x)=xf'(x)-f(x),\qquad\qquad
g_2(x)=-f'(x)
\end{array}
\end{equation}
for some function $f\in\ccF(I;c_1,c_2)$.
This definition coincides with the one of \cite[\S4]{Beresnevich-Dickinson-Velani-07:MR2373145}. Geometrically, the vector $(g_1(x),g_2(x),1)\in\R^3$ is define to be the cross-product of $(1,x,f(x))$ and $(0,1,f'(x))$ and thus is orthogonal to the latter two vectors.

Given positive real numbers $\delta$, $K$, $T$ and a subinterval $J\subset I$, let $B_{\vv g}(J,\delta,K,T)$ denote the set of $x \in J$ for which there exists $(q, p_1, p_2) \subset \Z^3 \setminus \{0\}$ such that
\begin{equation} \label{syst_g_general}
\left\{\begin{array}{l}
|q g_1(x) + p_1 g_2(x) + p_2| \leq \delta,\\[0.5ex]
|q g'_1(x) + p_1 g'_2 (x)| \leq K,\\[0.5ex]
|q| \leq T.
\end{array}
\right.
\end{equation}

\begin{theorem}\label{t:09}
Let $c_1$ and $c_2$ be positive constants, $I \subset \R$ be a compact interval and
\begin{equation}\label{lI}
L\eqdef\max_{x\in I}|x|.
\end{equation}
Then for any $f\in\ccF(I;c_1,c_2)$, any interval $J \subseteq I$ and any choice of $\delta,K,T$ satisfying
\begin{equation}\label{deltaKT}
0 < \delta \leq 1,\quad K>0,\quad T>1,\quad \delta K T \leq 1
\end{equation}
we have that
\begin{equation} \label{theo14_estimation}
|B_{\vv g}(J,\delta,K,T)| \leq E\,(\delta K T)^{\frac16}|J|,
\end{equation}
where $\vv g$ is given by \eqref{def_f},
\begin{equation}\label{E}
    E\eqdef\frac{648\,C}{\sqrt{\rho}},
\end{equation}
\begin{equation}\label{C}
    C=\max\Big\{C_0\sqrt{32},\,24\sqrt{6C_0M}\,\Big\},\qquad C_0=\frac{4c_2}{c_1},\qquad M=\sqrt{1+4L^2},
\end{equation}
\begin{equation}\label{rho}
\rho=\min\left\{1,c_1,\frac{c_1|J|\theta}{32M}\max\left\{\frac{16}{\delta},\ \frac{|J|}{K}\right\},\frac{c_1^2|J|^2T}{32\theta}\right\},\qquad \theta=(\delta KT)^{\frac13}.
\end{equation}
\end{theorem}

\begin{remark}\rm
For $T$ sufficiently large the constant $E$ appearing in (\ref{theo14_estimation}) is determined by $c_1,c_2$ and $L$ only and thus is independent from $J,\delta,K,T$. In order to see this, use the inequality $\max(x,y)\ge (xy)^{1/2}$ valid for all positive $x,y$ to get
$$
 \frac{c_1|J|\theta}{32M}\max\left\{\frac{16}{\delta},\ \frac{|J|}{K}\right\}\ge \frac{c_1|J|\theta}{32M}\left(\frac{16|J|}{\delta K}\right)^{1/2}\stackrel{\eqref{rho}}{=} \frac{c_1|J|\theta}{32M}\left(\frac{16|J|T}{\theta^3}\right)^{1/2} = \frac{c_1|J|^{3/2}T^{1/2}}{8M\theta^{1/2}}.
$$
Then, since $\theta<1$, it immediately becomes clear that $\rho=\min\{1,c_1\}$ and so
\begin{equation}\label{E1}
E=\frac{648\,C}{\min(1,\sqrt{c_1})}\qquad \text{when}\qquad T\ge \max\{64M^2|J|^{-3},\,32c_1^{-1}|J|^{-2}\}.
\end{equation}
Furthermore, in the case $\delta\le K$ we have a better estimate for $T$ in terms of $|J|$. To see this, note that
$\delta^2\le\delta K=\theta^3T^{-1}\le \theta^2T^{-1}$ and so $\delta\le \theta T^{-1/2}$. Then
$$
 \frac{c_1|J|\theta}{32M}\max\left\{\frac{16}{\delta},\ \frac{|J|}{K}\right\}\ge \frac{16c_1|J|\theta T^{1/2}}{32M\theta}= \frac{c_1|J|T^{1/2}}{2M}
$$
and one readily computes that
\begin{equation}\label{E2}
E=\frac{648\,C}{\min(1,\sqrt{c_1})}\qquad \text{when}\qquad T\ge \frac{\max\{4M^2,\,32c_1^{-2}\}}{|J|^{2}}\quad\text{and}\quad \delta\le K.
\end{equation}
\end{remark}

\begin{remark}\rm
Theorem~\ref{t:09} is the main stepping stone to the proof of Theorem~\ref{t:02}. Furthermore, its value is not limited to this application. For example, Theorem~\ref{t:09} can be used to extended the main result of  \cite{Bernik-Kleinbock-Margulis-01:MR1829381} (due to Bernik, Kleinbock and Margulis) to the set of weakly non-degenerate planar curves. Yet another application lies within the results of \cite{Beresnevich-Bernik-Goetze-09} (due to Beresnevich, Bernik and G\"otze) on the distribution of close conjugate algebraic numbers which can now be improved towards full effectiveness in the case of quadratic and integer cubic algebraic numbers.
\end{remark}

\begin{remark}\label{prolongation_I}\rm
Despite the fact that the functions $f\in\cF(I;c_1,c_2)$ (and consequently $\vv g$) are initially defined on the interval $I$ only we can always treat them as $C^2$ functions defined on the whole real line preserving condition~(\ref{cond_g2p}). Indeed, let $
T_2(z,x)\eqdef\sum_{i=0}^2\frac{1}{i!}f^{(i)}(z)(x-z)
$
denote the Taylor polynomial of degree $2$ and
consider the auxiliary function
$$
\tilde{f}(x)=\left\{\begin{array}{ccl}
T_2(x_1,x)& \text{ if } & x<x_1,\\[0.5ex]
f(x) & \mbox{ if } & x_1\leq x\leq x_2,\\[0.5ex]
T_2(x_2,x)& \text{ if } & x_2<x,
\end{array}
\right.
$$
where $[x_1,x_2]=I$. It is then easily verified that $\tilde f$ is $C^2(\R)$, satisfies (\ref{cond_g2p}) for all $x\in\R$ and coincides with $f$ on $I$. Hence the above claim follows.
\end{remark}

\subsection{$(C,\alpha)$-good functions}

The property of being $(C,\alpha)$-good introduced in \cite{Kleinbock-Margulis-98:MR1652916} by Kleinbock and Margulis lies at the heart of the proof of Theorem~\ref{t:09}. In this subsection we recall the key definition and various auxiliary statements from \cite{Bernik-Kleinbock-Margulis-01:MR1829381} and \cite{Kleinbock-Margulis-98:MR1652916}. We also establish a new lemma that provides sufficient conditions for a function to be $(C,\alpha)$-good -- Lemma~\ref{C-alpha-good} below.

Let $C$ and $\alpha$ be positive numbers and $V$ be a subset of $\R^d$. The function $f : V \rightarrow \R$ is said to be\/ \emph{$(C,\alpha)$-good on $V$}\/ if for any open ball
$B \subset V$ and any $\varepsilon > 0$ one has
\begin{equation} \label{def_Calpha_good}
\big|
\big\{x \in B \,:\, |f(x)| < \varepsilon \cdot \sup_{x\in B}|f(x)|\,\big\}
\big|
\leq C\varepsilon^{\alpha}|B| .
\end{equation}
Here, as before, $|A|$ denotes the Lebesgue measure of $A\subset\R^d$. Within this paper we shall only use the above definition in the case $d=1$. Several elementary properties of $(C, \alpha)$-good functions are now recalled.

\begin{lemma}[Lemma~3.1 in \cite{Bernik-Kleinbock-Margulis-01:MR1829381}]\label{lemma31}{\bf:}
\begin{itemize}
  \item[{\rm(a)}] If $f$ is $(C, \alpha)$-good on $V$, then so is $\lambda f$ for any $\lambda\in\R;$
\item[{\rm(b)}] If $f_1,\dots,f_k$ are $(C, \alpha)$-good on $V$, then so is $\max_{1\le i\le k} |f_i|;$
\item[{\rm(c)}]If $f$ is $(C, \alpha)$-good on $V$ and $c_1 \leq \frac{|f(x)|}{|g(x)|} \leq c_2$ for all $x \in V$, then $g$ is $\left(C(c_2/c_1)^{\alpha}, \alpha\right)$-good on $V;$
\item[{\rm(d)}] If $f$ is $(C, \alpha)$-good on $V$, then $f$ is $(C', \alpha')$-good on $V'$ for every $C'\geq C$, $\alpha' \leq \alpha$ and $V'\subset V$.
\end{itemize}
\end{lemma}

\begin{lemma}[Lemma~3.2 in \cite{Kleinbock-Margulis-98:MR1652916}] \label{lemma32}
For any $k \in \N$, any polynomial $f \in \R[x]$
of degree not greater than $k$ is $\left(2k(k + 1)^{1/k}, 1/k\right)$-good on\/ $\R$.
\end{lemma}

Before presenting the final lemma of this subsection, the following two technical statements are established.

\begin{lemma}\label{teclem}
Let $J$ be an interval, $\lambda>0$ and $\theta:J\to\R$ be a $C^1$ function such that $\inf_{x\in J}|\theta'(x)|\ge \lambda$. Then \ $\sup_{x\in J}|\theta(x)|\ge \tfrac12\lambda|J|$.
\end{lemma}

\begin{proof}
Let $y_1$ and $y_2$ be the endpoints of $J$. Then,
by the Mean Value Theorem, for any sufficiently small $\ve>0$ we have that
$\left|\theta(y_1+\ve)-\theta(y_2-\ve)\right|\geq \ \lambda (y_2-y_1-2\ve)=\lambda(|J|-2\ve)$. Therefore,
$|\theta(y_1+\ve)|+|\theta(y_2-\ve)|\geq \ \lambda (y_2-y_1-2\ve)=\lambda(|J|-2\ve)$. Hence
\begin{equation}\label{e2}
2\sup_{x\in J}\left|\theta(x)\right|\geq\ \lambda (|J|-2\ve).
\end{equation}
For $\ve>0$ is arbitrarily, (\ref{e2}) implies $\sup_{x\in J}|\theta(x)|\ge \tfrac12\lambda|J|$ and completes the proof.
\end{proof}

\begin{lemma}\label{teclem3}
Let $f$ be a $C^1$ function on an interval $B$ such that
\begin{equation} \label{sup_pg_2inf+}
    \sup_{x\in B}|f(x)| \ \ge \ 2\inf_{x\in B}|f(x)|.
\end{equation}
Then
\begin{equation}\label{sup_majorated_2int+}
    \ \ \ \sup_{x\in B}|f(x)| \ \le \  2\int_{B}|f'(x)|dx.
\end{equation}
\end{lemma}

\begin{proof}
By the Fundamental Theorem of Calculus, for any $y_1,y_2\in B$ we have that $f(y_2)=f(y_1)+\int_{y_1}^{y_2}f'(x)dx$.
Then, using triangle inequality gives $|f(y_2)|\le|f(y_1)|+\int_{B}|f'(x)|dx$.
Further, taking supremum over $y_2\in B$ and infimum over $y_1\in B$ gives the estimate
$$
    \sup_{x\in B}|f(x)|\leq\inf_{x\in B}|f(x)|+\int_{B}|f'(x)|dx
    \ \stackrel{\eqref{sup_pg_2inf+}}{\le}\ \tfrac12\sup_{x\in B}|f(x)|+\int_{B}|f'(x)|dx
$$
which readily implies (\ref{sup_majorated_2int+}).
\end{proof}

\begin{lemma}\label{C-alpha-good}
Let $\kappa_0$, $\kappa_1$ and $\kappa_2$ be some positive constants and $f$ be a $C^1$ function on an interval $I$ such that for any subinterval $B\subset I$
\begin{equation} \label{cond1}
    \sup_{x\in B}|f(x)| \ \ge \kappa_0|B|^2
\end{equation}
and
\begin{equation}\label{cond2}
\sup_{x\in B}|f'(x)|\leq \kappa_1\inf_{x\in B}|f'(x)| + \kappa_2|B|.
\end{equation}
Assume that $f'$ has at most $r$ roots in $I$.
Then $f$ is $(C_\kappa,\tfrac12)$-good on $I$ with $$C_\kappa\eqdef\max\left\{4,\ \dfrac{4(\kappa_1\kappa+\kappa_2)}{\kappa}, \ (r+1)\sqrt{\dfrac{2(\kappa_1\kappa+\kappa_2)}{\kappa_0}}\ \right\},\quad \kappa>0\,.
$$
\end{lemma}

\begin{proof}
Fix any $\kappa>0$. By definition, in order to prove that $f$ is $(C_\kappa,\frac12)$-good on $I$ we have to verify that for any interval $B\subset I$ the set
\begin{equation} \label{vbtecset1}
B_\ve\eqdef\big\{x \in B \,:\, |f(x)| < \varepsilon \cdot \sup_{x\in B}|f(x)|\big\}
\end{equation}
satisfies
\begin{equation} \label{vbCa_good_eta}
|B_\ve| \leq C_\kappa\,\varepsilon^{\frac12}|B| .
\end{equation}
Since $C_\kappa\geq 4$, for any $\varepsilon\geq \frac12$ the r.h.s. of~(\ref{vbCa_good_eta}) does not fall below $|B|$ and (\ref{vbCa_good_eta}) is a priori true. Henceforth, we may assume that $\ve<\tfrac12$. Consequently, if
$\sup_{x\in B}|f(x)| \le 2\inf_{x\in B}|f(x)|$
then the set in the l.h.s. of~(\ref{vbCa_good_eta}) is empty and (\ref{vbCa_good_eta}) is trivially satisfied. Otherwise, by Lemma~\ref{teclem3}, we have the inequality
\begin{equation}\label{vbsup_majorated_2int}
    \sup_{x\in B}|f(x)| \le 2\int_{B}|f'(x)|dx
\end{equation}
assumed for the rest of the proof, which will depend upon the magnitude of $\lambda\eqdef\inf_{x\in B}|f'(x)|$.
By (\ref{cond2}),
$
\sup_{x\in B}|f'(x)|\leq \kappa_1\lambda + \kappa_2|B|.
$
Combining the latter inequality with (\ref{vbsup_majorated_2int}) gives
\begin{equation} \label{vbsup_majorated_lambda}
    \sup_{x\in B}|f(x)| < 2\big(\kappa_1\lambda + \kappa_2|B|\big)|B|.
\end{equation}

\noindent\textbf{Case (i).} Assume that $\lambda\geq \kappa|B|$. Then $f'$ does not change sign on $B$ and hence $f$ is monotonic. Therefore, for any $\tau>0$, the set $J=\big\{x \in B \,:\, |f(x)| < \tau\big\}$ is an interval (possibly empty). Then, by Lemma~\ref{teclem}, $\tau\ge \tfrac12\lambda|J|$, that is
$\big|
\big\{x \in B \,:\, |f(x)| < \tau\big\}
\big|
\leq \frac{2\tau}{\lambda}$.
Taking $\tau=\varepsilon\cdot\sup_{x\in B}|f(x)|$ ensures that $J=B_\ve$ and gives
$$
|B_\ve| \leq \frac{2\varepsilon \cdot \sup_{x\in B}|f(x)|}{\lambda} \ \stackrel{\eqref{vbsup_majorated_lambda}}{<} \ \frac{4\varepsilon\big(\kappa_1\lambda + \kappa_2|B|\big)|B|}{\lambda}.
$$
By the hypothesis that $\lambda\geq \kappa|B|$, we have that
$\frac{\left(\kappa_1\lambda+\kappa_2|B|\right)}{\lambda}\leq \frac{\left(\kappa_1\kappa+\kappa_2\right)}{\kappa}$
and further obtain the required estimate
$$
|B_\ve|\ \leq \ \frac{4\left(\kappa_1\kappa+\kappa_2\right)}{\kappa}\,\varepsilon\,|B|\ \stackrel{(\ve<\frac12)}{<}\
\frac{4\left(\kappa_1\kappa+\kappa_2\right)}{\kappa}\, \varepsilon^{1/2}\, |B|\ \le \ C_\kappa\, \varepsilon^{1/2}\, |B|\,.
$$

\smallskip

\noindent\textbf{Case (ii).} Assume that $\lambda< \kappa|B|$. Since $f'(x)$ has at most $r$ roots in $B$, the interval $B$ can be split into at most $(r+1)$ subintervals such that $f$ is monotonic on each of them. Consequently, for every $0<\ve<\tfrac12$ the set $B_\ve$ is the union of at most $r+1$ intervals. Let $B'$ denote the biggest interval in $B_\ve$. Then
\begin{equation} \label{vbtecset1+}
    |B_\ve|\leq (r+1)|B'|.
\end{equation}
By (\ref{cond1}),
\begin{equation}\label{vbvb8}
\sup_{x\in B'}|f(x)|\ge \kappa_0|B'|^2.
\end{equation}
On the other hand, by the definition of $B'$, we have that
\begin{equation}\label{vbvb9}
\sup_{x\in B'}|f(x)|\le \ve\cdot\sup_{x\in B}|f(x)|\ \stackrel{\eqref{vbsup_majorated_lambda}}{\le} \ 2\ve \big(\kappa_1\lambda + \kappa_2|B|\big)|B|\ \ \stackrel{\lambda<\kappa|B|}{\le} \ 2\ve \big(\kappa_1\kappa + \kappa_2\big)|B|^2.
\end{equation}
Comparing (\ref{vbvb8}) and (\ref{vbvb9}) gives a bound on $|B'|$, which together with (\ref{vbtecset1+}) establishes (\ref{vbCa_good_eta}) and thus completes the proof.
\end{proof}

\subsection{Properties of certain families of functions}

In this subsection we investigate certain families of functions for being $(C,\alpha)$-good and other relevant properties. Unless otherwise stated $f$, $g_1$, $g_2$ and $I$ are the same as at the beginning of \S\ref{effBKM}. Although all the statements are established for $f$, in view of Remark~\ref{prolongation_I} (on page\,\pageref{prolongation_I}) they are true for $\tilde f$ and with $I$ replaced by any interval $\tilde I$, in particular, for $\tilde I=3^3 I$, which is of our main interest.

\begin{lemma}\label{linpol_Ca_good}
Let $f\in\cF(I;c_1,c_2)$ and let $L$ be given by \eqref{lI}. Define the constants
\begin{equation} \label{def_C0}
M=\sqrt{1+4L^2},\qquad C_0=4c_2c_1^{-1}\qquad\text{and}\qquad
C_1=2\max\{C_0,\sqrt{32C_0M}\}.
\end{equation}
Let $a,b \in \R$ satisfy $a^2+b^2\ge1$ and
$\eta: I \rightarrow \R$ satisfy
\begin{equation} \label{linpol_Ca_good_b_psi_definition}
\eta'(x)=(ax+b)f''(x),
\end{equation}
that is $\eta$ is an antiderivative of $(ax+b)f''(x)$. Then
\begin{list}{{\rm(\alph{tmpabcd})}}{\usecounter{tmpabcd}}
\item\label{linpol_Ca_good_a} $\eta'$ is $\left(C_0,1\right)$-good on $I$\/$;$
\item\label{linpol_Ca_good_a1}
for any interval $B \subset I$\qquad
\refstepcounter{equation}\label{vb1}
$\displaystyle\sup_{x \in B}|\eta'(x)| \geq \frac{c_1|B|}{2M}\,;$\hspace*{\fill}{\rm(\theequation)}
\item\label{linpol_Ca_good_b} $\eta$
is $(C_1,\frac12)$-good on $I$\/$;$
\item\label{linpol_Ca_good_c} for any interval $B \subset I$\qquad
\refstepcounter{equation}\label{vb5}
$\displaystyle\sup_{x \in B}|\eta(x)| \geq \frac{c_1|B|^2}{32M}\,.$\hspace*{\fill}{\rm(\theequation)}
\end{list}
\end{lemma}

\begin{proof}
Part~(\ref{linpol_Ca_good_a} is readily implied by Lemma~\ref{lemma32} and Lemma~\ref{lemma31}(c). Further,
to prove part (\ref{linpol_Ca_good_a1} we distinguish two cases.
If $|a| \leq \frac{|b|}{2L}$ then
$1\le a^2+b^2\leq b^2\left(\frac{1}{4L^2}+1\right)$. Consequently
$|b| \geq \frac{2L}{M}$ and therefore
\begin{equation} \label{linpol_Ca_good_inf_minoration}
\sup_{x \in B}|ax+b|\geq\inf_{x \in B}|ax+b|\geq|b|-|a|\sup_{x\in B}|x|\geq |b|-|a|L\geq\frac{|b|}{2} \geq\frac{L}{M}.
\end{equation}
By (\ref{lI}), $|B|\le2L$. Therefore, (\ref{linpol_Ca_good_inf_minoration}) together with (\ref{cond_g2p}) implies \eqref{vb1} in the case $|a| \leq \frac{|b|}{2L}$. Otherwise, $|a| > \frac{|b|}{2L}$ and then
$1\le a^2+b^2 < a^2(1+4L^2)=(aM)^2$. It follows that
$|a|>\frac{1}{M}$. Therefore,
\begin{equation}\label{vb2}
\sup_{x \in B}|ax+b|=|a|\,\sup_{x \in B}|x-(-b/a)|\geq |a|\,\dfrac{|B|}2\ge \frac{1}{2M}|B|.
\end{equation}
The latter together with (\ref{cond_g2p}) implies \eqref{vb1} and thus completes the proof of part~(\ref{linpol_Ca_good_a1}.

\medskip

We now prove part~(\ref{linpol_Ca_good_c}. As with the proof of part (b) we distinguish the following two cases:
$|a| \leq \frac{|b|}{2L}$ and  $|a| > \frac{|b|}{2L}$.
In the first case, by (\ref{linpol_Ca_good_inf_minoration}) and the inequality $|B|\le2L$ implied by (\ref{lI}), we have that $|\eta'(x)| \geq \frac{c_1|B|}{2M}$ for all $x\in B$.
By Lemma~\ref{teclem}, we then obtain (\ref{vb5}).
In the second case, split $B$ into three subintervals: $B_l$, $B_m$ and $B_r$, where $B_m$ is the middle half of $B$; $B_l$ is the left quarter of $B$; and $B_r$ is the right quarter of $B$. Then applying (\ref{vb2}) to $B_m$ gives
\begin{equation}\label{vb3}
\sup_{x \in B_m}|ax+b|\ge \frac{|B_m|}{2M}=\frac{|B|}{4M}.
\end{equation}
In the case under consideration $a\not=0$ and therefore the function $ax+b$ is strictly monotonic. It follows that the supremum in (\ref{vb3}) is attained at the endpoints of $B_m$. Hence, either $\inf_{x\in B_l}|ax+b|$ \,or\, $\inf_{x\in B_r}|ax+b|$ is bounded away from 0 by the right hand side of (\ref{vb3}). Then, by (\ref{cond_g2p}) and (\ref{linpol_Ca_good_b_psi_definition}), we have
\begin{equation}\label{vb4}
    \max\Big\{\inf_{x \in B_l}|\eta'(x)|,\
    \inf_{x \in B_r}|\eta'(x)|\Big\}\ge \frac{c_1|B|}{4M}.
\end{equation}
Applying Lemma~\ref{teclem} with $J$ being equal to either $B_l$ or $B_r$, $\theta=\eta$ and $\lambda=$\,\,r.h.s.\,of\,\,(\ref{vb4}) we again obtain (\ref{vb5}).

\medskip

Finally, in order to prove part~(\ref{linpol_Ca_good_b} we will appeal to Lemma~\ref{C-alpha-good} with $f=\eta$. In view of Lemma~\ref{lemma31}(a), without loss of generality we can assume that $a^2+b^2=1$ and so $|a|\le1$. By (\ref{linpol_Ca_good_b_psi_definition}), for any  $y_1,y_2\in B$, we have that
$
\eta'(y_2)=\eta'(y_1)\frac{f''(y_2)}{f''(y_1)}+a(y_2-y_1)f''(y_2).
$
Then, using the triangle inequality, inequalities (\ref{cond_g2p}), $|y_2-y_1|\le |B|$ and $|a|\le 1$ we get
$    |\eta'(y_2)|\leq\frac{c_2}{c_1}|\eta'(y_1)|+c_2|B|$. Taking supremum over $y_2\in B$ and infimum over $y_1\in B$ gives (\ref{cond2}) with $\kappa_1=c_2/c_1$ and $\kappa_2=c_2$. Also (\ref{vb5}) verifies (\ref{cond1}) with $\kappa_0=c_1/(32M)$. It is also easily seen that $r=1$. Then, by Lemma~\ref{C-alpha-good} with $\kappa=c_1$ we get the conclusion of part~(c).
\end{proof}

\bigskip

If $g_1$ and $g_2$ are given by (\ref{def_f}), then $\eta(x)=c+ag_1(x)+bg_2(x)$, where $a,b,c\in\R$, satisfies
$\eta'(x)=ag_1'(x)+bg_2'(x)=(ax-b)g''(x)$. Therefore, by Lemma~\ref{linpol_Ca_good}, we get the following

\begin{corollary}\label{cor_linpol_Ca_good}
Let $g_1$ and $g_2$ be given by \eqref{def_f}, $a,b,c \in \R$ and $a^2+b^2\ge 1$. Let $C_0,C_1$ and $M$ be the same as in Lemma~\ref{linpol_Ca_good}. Then
\begin{list}{{\rm(\alph{tmpabcd})}}{\usecounter{tmpabcd}}
\item \label{cor_linpol_Ca_good_a} $ag_1'+bg_2'$ is $\big(C_0,1\big)$-good on $I$\/{\rm;}
\item \label{cor_linpol_Ca_good_b}
$
\displaystyle \sup_{x \in B}|ag_1'(x)+bg_2'(x)| \geq \frac{c_1|B|}{2M}
$ \ \ for any subinterval $B \subset I${\rm;}
\item \label{cor_linpol_Ca_good_c}  $ag_1+bg_2+c$ is $(C_1,\frac12)$-good on $I$\/{\rm;}
\item \label{cor_linpol_Ca_good_d}
$
\displaystyle
\sup_{x \in B}|ag_1(x)+bg_2(x)+c| \geq \frac{c_1|B|^2}{32M}
$ \ \ for any subinterval $B \subset I$.
\end{list}
\end{corollary}

\begin{lemma}\label{skewdet_reduced}
Let $a,b \in \R$ and $\tilde\eta(x)\eqdef(f(x)+ax+b)f''(x)$ for $x\in I$. Let $C_2=C_0\sqrt{32}$. Then
\begin{list}{{\rm(\alph{tmpabcd})}}{\usecounter{tmpabcd}}
\item\label{skewdet_reduced_a} $\tilde\eta$ is $(C_2,\frac12)$-good on $I;$
\item \label{skewdet_reduced_b} for any interval $B \subset I$\qquad
\refstepcounter{equation}\label{skewdet_reduced_resb}
$\displaystyle\sup_{x \in B}|\tilde\eta(x)| \geq \frac{c_1^2}{32}|B|^2\,.$\hspace*{\fill}{\rm(\theequation)}
\end{list}
\end{lemma}

\begin{proof} Let $\theta(x)=f(x)+ax+b$. Clearly $\theta''(x)=f''(x)$.
Split $B$ into three subintervals: $B_l$, $B_m$ and $B_r$, where $B_m$ is the middle half of $B$; $B_l$ is the left quarter of $B$; and $B_r$ is the right quarter of $B$. By (\ref{cond_g2p}) and Lemma~\ref{teclem}, $\sup_{x\in B_m}|\theta'(x)|\ge \tfrac{c_1}2|B_m|=\tfrac{c_1}4|B|$. In view of (\ref{cond_g2p}), $\theta'$ is monotonic. Therefore,
$\inf_{x\in B'}|\theta'(x)|\ge \tfrac{c_1}4|B|$ for at least one choice of $B'$ from $B_r$ and $B_l$. Applying Lemma~\ref{teclem} to $B'$ further gives the estimate
\begin{equation}\label{vb10}
\sup_{x\in B'}|\theta(x)|\ge\tfrac12|B'|\frac{c_1}{4}|B|=\frac{c_1}{32}|B|^2.
\end{equation}
Using the identity $\tilde\eta(x)=\theta(x)f''(x)$ and (\ref{cond_g2p}) gives (\ref{skewdet_reduced_resb}).

Further, for any $y_1,y_2\in B$ we have $\theta'(y_2)=\theta'(y_1)+(f'(y_2)-f'(y_1))$. By the Mean Value Theorem and (\ref{cond_g2p}), we have $|f'(y_2)-f'(y_1)|\le c_2|B|$. Further, with reference to the former equality, taking supremum over $y_2\in B$ and infimum over $y_1\in B$ gives (\ref{cond2}) with $\kappa_1=1$ and $\kappa_2=c_2$. Also, (\ref{vb10}) ensures (\ref{cond1}) with $\kappa_0=c_1/32$. Applying Lemma~\ref{C-alpha-good} with $\kappa=c_2$ gives that $\theta$ is $(8\sqrt{2C_0},\tfrac12)$-good on $I$, where $C_0$ is defined by (\ref{def_C0}). Finally, since $\tilde\eta(x)=\theta(x)f''(x)$, using Lemma~\ref{lemma31}(c) and (\ref{cond_g2p}) establishes the statement of part~(a).
\end{proof}

\bigskip

Our final statement of this subsection is concerned with the skew-gradient of pairs of functions
\begin{equation}
\vv u\cdot\hg(x)\quad\text{with}\quad \vv u\in\Z^3\setminus\{0\},
\end{equation}
where $\hg(x)\eqdef(g_1(x),g_2(x),1)$ and $\vv u\cdot\hg(x)$ is the scalar product of $\vv u$ and $\hg(x)$. The skew-gradient of a pair of functions $(\gamma_1,\gamma_2)$ as defined in \cite[\S4]{Bernik-Kleinbock-Margulis-01:MR1829381} is given by
$$
\tilde\nabla(\gamma_1,\gamma_2)(x)\eqdef \gamma_1(x)\gamma_2'(x)-\gamma_1'(x)\gamma_2(x).
$$

\begin{lemma}\label{cor_det_skew_reduced}
Let $C_2=C_0\sqrt{32}$ be the same as in Lemma~\ref{skewdet_reduced}. Then for any $\vv u,\vv w\in\Z^3\setminus\{0\}$ with $\vv u\wedge\vv w\not=0$
\begin{list}{{\rm(\alph{tmpabcd})}}{\usecounter{tmpabcd}}
\item \label{cor_det_skew_reduced_a} $\SD (\vv u\cdot\hg,\vv w\cdot\hg)$ is $(C_2,\frac12)$-good on $I$.
\item \label{cor_det_skew_reduced_b} for any interval $B \subset I$\qquad\refstepcounter{equation}\label{ie_rho2}
$\displaystyle\sup_{x\in B}|\SD (\vv u\cdot\hg,\vv w\cdot\hg)(x)| \geq \min\left\{\frac{c_1^2|B|^2}{32},\frac{c_1|B|}{2M}\right\}.$\hspace*{\fill}{\rm(\theequation)}
\end{list}
\end{lemma}

\begin{proof}
Given any $\vv u,\vv w\in\Z^3\setminus\{0\}$ such that $\vv u\wedge\vv w\not=0$, by the Laplace identity (see, e.g., \cite[Lemma~6D, p.105]{Schmidt-1980}),
\begin{equation}\label{vb-12}
 \SD (\vv u\cdot\hg,\vv w\cdot\hg)(x)\eqdef\left|\begin{array}{cc}
                                              \vv u\cdot\hg(x) & \vv w\cdot\hg(x) \\
                                              \vv u\cdot\hg'(x) & \vv w\cdot\hg'(x)
                                            \end{array}
 \right|=(\hg(x)\wedge\hg'(x))\cdot(\vv u\wedge\vv w).
\end{equation}
By (\ref{def_f}) and the definition of exterior product, one easily verifies that $\hg(x)\wedge\hg'(x)=f''(x)(f(x),-x,1)$. Also $\vv u\wedge\vv w=(p,q,r)\in\Z^3\setminus\{0\}$. Thus, by (\ref{vb-12}),
\begin{equation}\label{vb+}
    \SD (\vv u\cdot\hg,\vv w\cdot\hg)(x)=f''(x)(pf(x)-qx+r).
\end{equation}
If $p=0$ then Lemma~\ref{linpol_Ca_good}(a) implies Lemma~\ref{cor_det_skew_reduced}(a) via Lemma~\ref{lemma31}(d);
Lemma~\ref{linpol_Ca_good}(b) implies Lemma~\ref{cor_det_skew_reduced}(b). If $p\not=0$ then $|p|\ge1$. By this fact and Lemma~\ref{lemma31}, without loss of generality we can assume that $p=1$ as otherwise we would divide (\ref{vb+}) through by $p$. In this case Lemma~\ref{skewdet_reduced} with $\tilde\eta=\SD(\vv u\cdot\hg,\vv w\cdot\hg)(x)$ completes the proof.
\end{proof}

\subsection{Proof of Theorem~\ref{t:09}}
Theorem~\ref{t:09} will be derived from a general result due to Kleinbock and Margulis appearing as
Theorem~5.2 in \cite{Kleinbock-Margulis-98:MR1652916}.
In order to state it we recall some notation from \cite{Kleinbock-Margulis-98:MR1652916}. In what follows $\cC(\Z^{k})$ will denote the set of all non-zero complete sublattices of $\Z^k$. An integer lattice $\Lambda\subset\Z^k$ is called \emph{complete}\/ if it contains all integer points lying in the linear space generated by $\Lambda$. Given a lattice $\Lambda\subset\R^k$ and a basis $\vv w_1,\dots,\vv w_r$ of $\Lambda$, the multivector $\vv w_1\wedge\dots\wedge\vv w_r$ is uniquely defined up to sign since any two basis of $\Lambda$ are related by a unimodular transformation. Therefore, the following function on the set of non-zero lattices is
well defined:
\begin{equation}\label{e:090}
\|\Lambda\|\eqdef|\vv w_1\wedge\dots\wedge\vv w_r|_\infty\,,
\end{equation}
where $|\cdot|_\infty$ denotes the supremum norm on $\bigwedge(\R^k)$.

\begin{theoremKM}[Theorem~5.2 in \cite{Kleinbock-Margulis-98:MR1652916}]
Let $d,k\in\N$, $C,\alpha>0$ and $0<\rho\le 1$ be given. Let $B$ be a
ball in $\R^{d}$ and $h:3^kB \to \operatorname{GL}_k(\R)$ be given. Assume that for
any $\Lambda \in \cC(\Z^{k})$
\begin{enumerate}
\item[{\rm(i)}] the function $\vv x\mapsto \|h(\vv x)\Lambda\|$ is
$(C,\alpha)$-good\ on $3^{k}B$, and
\item[{\rm(ii)}] $\sup_{\vv x\in B}\|h(\vv x)\Lambda\|\ge\rho$.
\end{enumerate}
Then there is a constant $N_{d}$ depending on $d$ only such that for
any $ \ve >0$ one has
\begin{equation}\label{est1}
\big|\big\{\vv x\in B:\min_{\vv a\in\Z^k\setminus\{0\}}|h(\vv x)\vv a|_\infty \le
\ve\big\}\big| \le kC (3^{d}N_{d})^{k} \left(\frac\ve \rho
\right)^\alpha |B|.
\end{equation}
\end{theoremKM}

We will use this general result in the case $k=3$ and $d=1$. Note that the Besicovitch constant $N_d$ appearing in (\ref{est1}) equals 2 in the case $d=1$.

\bigskip

\begin{proofthree}
Since the right hand side of (\ref{theo14_estimation}) is independent of $f$ and $\cF(I;c_1,c_2)$ is dense in $\ccF(I;c_1,c_2)$ (in the uniform convergence topology), it suffices to prove Theorem~\ref{t:09} for $f\in\cF(I;c_1,c_2)$.
Let $J\subset I$ be an interval and $\theta=(\delta KT)^{1/3}$. Define
$$
t_1=\frac{\theta}{\delta},\qquad t_2=\frac{\theta}{K},\qquad t_3=\frac{\theta}{T},
$$
\begin{equation} \label{theo51_defUx}
g_{\vv t}=\diag(t_1,t_2,t_3)\qquad\text{and}\qquad
G_x=\left(\begin{array}{ccc}
g_1(x)  & g_2(x) & 1\\
g_1'(x) & g_2'(x) & 0\\
1 & 0 & 0
\end{array}\right).
\end{equation}
It easily follows from the above definitions that
\begin{equation}\label{e:088}
B(J,\delta,K,T)=\big\{x\in J:\min_{\vv a\in\Z^3\setminus\{0\}}|h(x)\vv a|_\infty \le
\theta\,\big\}\,,\qquad\text{where $ h(x)=g_{\vv t}G_x $}.
\end{equation}
It is also readily seen that $\det g_{\vv t}=1$ and
\begin{equation}\label{e:093}
\det h(x)=\det G_x= -g'_2(x)= -f''(x) \stackrel{\eqref{cond_g2p}}{\not=}0.
\end{equation}
Therefore, $h(x)\in \operatorname{GL}_{3}(\R)$. In view of Remark~\ref{prolongation_I} we will regard $h$ as a map defined on $3^3J$.

Our next goal is to verify conditions $(i)$ and $(ii)$ of Theorem~KM
for the specific choice of $h$ made by (\ref{e:088}). Fix a
$\Gamma\in\cC(\Z^3)$. Let $r=\dim\Gamma>0$. We will consider the three cases $r=1,2,3$ separately. It is easily seen that $C$ defined by (\ref{C}) satisfies $C=\max\{C_0,C_1,\tilde C_2\}$, where $C_0$ and $C_1$ are defined by (\ref{def_C0}) and $\tilde C_2$ is defined in the same way as $C_2$ within Lemma~\ref{skewdet_reduced} but with $M$ replaced by $27M$.

\noindent\textbf{Case (1).} Let $r=1$. Then the basis of $\Gamma$ consists of just one integer vector, say $\vv w={}^t(w_1,w_2,w_3)\not=0$ (here and elsewhere ${}^t$ denotes transposition). Consequently, $h(x)\vv w$ is a basis of $h(x)\Gamma$ and $\|h(x)\Gamma\|=|h(x)\vv w|_\infty$. Using (\ref{theo51_defUx}) we get
\begin{equation}\label{vb+1}
    h(x)\vv w=\left(\begin{array}{c}
                       t_1\vv w\cdot\hg(x)\\
                       t_2\vv w\cdot\hg'(x)\\
                       t_3w_1
                    \end{array}
    \right),
\end{equation}
where $\vv w\cdot\hg(x)=w_1g_1(x)+w_2g_2(x)+w_3$ is the scalar product of $\vv w$ and $\hg(x)={}^t(g_1(x),g_2(x),1)$.

\noindent\textbf{Subcase (1a).} Assume that $w_1=w_2=0$. Then $h(x)\vv w={}^t(1,0,0)$ and it is easily verified using (\ref{def_Calpha_good}) and Lemma~\ref{lemma31}(b) that $\|h(x)\Gamma\|$ is $(C,1/2)$-good on $3^3J$. It is also clear that
\begin{equation}\label{vb-15}
\|h(x)\Gamma\|=1\qquad\text{if }w_1^2+w_2^2=0.
\end{equation}

\noindent\textbf{Subcase (1b).} Assume that $(w_1,w_2)\not=\vv0$. Since $w_1,w_2\in\Z$, we have $w_1^2+w_2^2\ge 1$.
Then using Corollary~\ref{cor_linpol_Ca_good}(a)+(c) and Lemma~\ref{lemma32} we verify that every coordinate function in (\ref{vb+1}) is $(C,\frac12)$-good on $3^3J$. Consequently, by Lemma~\ref{lemma31}(b), $\|h(x)\Gamma\|$ is $(C,\frac12)$-good on $3^3J$. Further, applying Corollary~\ref{cor_linpol_Ca_good}(b)+(d) to the first and second coordinate functions in (\ref{vb+1}) gives
\begin{equation}\label{vb+2}
    \sup_{x\in J}\|h(x)\Gamma\| \ \ge \ \max\left\{\frac{t_1c_1|J|}{2M},\ \frac{t_2c_1|J|^2}{32M}\right\}\qquad\text{if }w_1^2+w_2^2\not=0
\end{equation}

\noindent\textbf{Case (2).} Let $r=2$. Then the basis of $\Gamma$ consists of two integer vectors, say $\vv u=(u_1,u_2,u_3)$ and $\vv w=(w_1,w_2,w_3)$ with $\vv u\wedge\vv w\not=0$. Consequently, $h(x)\vv u$ and $h(x)\vv w$ is a basis of $h(x)\Gamma$ and $\|h(x)\Gamma\|=|h(x)\vv u\wedge h(x)\vv w|_\infty$. Using (\ref{vb+1}) and a similar expression for $h(x)\vv u$ one readily verifies that
\begin{equation}\label{vb+3}
    h(x)\vv u\wedge h(x)\vv w=\left(\begin{array}{c}
                       t_1t_2\SD(\vv u \cdot\hg,\vv w\cdot\hg)(x)\\
                       t_1t_3(w_1\vv u-u_1\vv w)\cdot\hg(x)\\
                       t_2t_3(w_1\vv u-u_1\vv w)\cdot\hg'(x)
                    \end{array}
    \right).
\end{equation}
Using Corollary~\ref{cor_linpol_Ca_good}(a)+(c) and Lemma~\ref{cor_det_skew_reduced}(a) we immediately verify  $\|h(x)\Gamma\|$ is $(C,\frac12)$-good on $3^3J$. Further, by Lemma~\ref{cor_det_skew_reduced}(b),
\begin{equation}\label{vb+4}
    \sup_{x\in J}\|h(x)\Gamma\|\ge \sup_{x\in J}|t_1t_2\SD(\vv u \cdot\hg,\vv w\cdot\hg)(x)|\ge t_1t_2\frac{c_1^2}{32}|J|^2.
\end{equation}

\noindent\textbf{Case (3).} Let $r=3$. Then, $\Gamma=\Z^3$. Consequently,
\begin{equation}\label{vb-16}
\|h(x)\Gamma\|=|\det h(x)|=|\det G_x|=|g''(x)|\ge c_1.
\end{equation}

\medskip

\noindent\textbf{Completion of the proof.} The upshot of (\ref{vb-15}), (\ref{vb+2}), (\ref{vb+4}) and (\ref{vb-16}) is that $\sup_{x\in J}\|h(x)\Gamma\|\ge \rho$, where $\rho$ is given by (\ref{rho}). Thus Theorem KM is applicable with this value of $\rho$, $\alpha=\tfrac12$ and $C$ given by (\ref{C}). Then, by (\ref{est1}) with $k=3$ and $d=1$ and (\ref{e:088}), we get (\ref{theo14_estimation}).
\end{proofthree}

\section{Proof of Theorem~\ref{t:02}}

We follow the proof of Theorem~7 in \cite{Beresnevich-Dickinson-Velani-07:MR2373145} replacing the use of Lemma~6 of~\cite{Beresnevich-Dickinson-Velani-07:MR2373145} with our Theorem~\ref{t:09}. Recall that, by Theorem~\ref{t:08}, it suffices to consider functions in $\cF(I;c_1,c_2)$ only. Thus we fix any $f\in\cF(I;c_1,c_2)$ and fix any non-empty interval $J \subseteq I$ of length $|J|\le\tfrac12$. Since the set of rational points and their denominators are invariant under translations by integer points, without loss of generality we can assume that $J\subset[-\tfrac12;\tfrac12]$.
Consequently, $M\le\sqrt2$ and $C\le36c_2/c_1$, where $M$ and $C$ are defined by (\ref{C}). Then
$$
\frac{648\,C}{\min(1,\sqrt{c_1})}\le \ \hat E
$$
where $\hat E$ is the constant defined by (\ref{mainconst}).
Furthermore, (\ref{vb21}) ensures that either (\ref{E1}) or (\ref{E2}) is applicable. Henceforth,
\begin{equation}\label{E3}
E\le\hat E.
\end{equation}
Let $\delta$ and $Q$ satisfy (\ref{e:003}) and (\ref{vb21}).
Let $\vv g$ be given by (\ref{def_f}) and let $B_{\vv g}(\ldots)$ be defined in the same way as in Theorem~\ref{t:09}, that is the set of $x \in J$ such that there exists a non-zero integer solution $(q,p_1,p_2)$ to (\ref{syst_g_general}).
By Theorem~\ref{t:09} and inequality (\ref{E3}),
\begin{equation}
|B_{\vv g}(J,c_0\delta,c_2(c_0Q\delta)^{-1},2 c_0Q)|\ \leq \ \hat E|J|(2c_0c_2)^{\frac16}\ \stackrel{\eqref{mainconst}}{=} \ \tfrac14|J|.
\end{equation}
Therefore the set $G\eqdef\tfrac34J \setminus B_{\vv g}(J,c_0\delta,c_2(c_0Q\delta)^{-1},2 c_0Q)$ satisfies
\begin{equation} \label{theo7_Bisgrand}
|G| \geq \tfrac12|J|,
\end{equation}
where $\frac34 J$ is the interval $J$ scaled by $\frac34$.
Take $x\in G$. By Minkowski's linear forms theorem, there is a coprime triple $(q,p_1,p_2)\in\Z^3 \setminus \{0\}$ satisfying the system of inequalities
\begin{equation}\label{vb+20}
\left\{\begin{array}{l}
|q g_1(x) + p_1 g_2(x) + p_2| \leq  c_0\delta \\[0.5ex]
|q g'_1(x) + p_1 g'_2 (x)| \leq c_2(c_0Q\delta)^{-1}\\[0.5ex]
0\le q \leq Q.
\end{array}\right.
\end{equation}
By the definition of $G$,
\begin{equation} \label{smallqlarge}
q > 2 c_0Q.
\end{equation}
By (\ref{def_f}) and the second inequality of~(\ref{vb+20}), we have that
$|qxf''(x)-p_1f''(x)|<c_2\left( c_0Q\delta \right)^{-1}$.
This together with~(\ref{smallqlarge}) and (\ref{cond_g2p}) implies that
\begin{equation}\label{vbvb}
\left|x-\frac{p_1}{q}\right|\leq\frac{c_2}{2c_1 c_0^2Q^2\delta}=\frac{C_1}{Q^2\delta}\stackrel{\eqref{e:003}}{\le} \tfrac18|J|.
\end{equation}
Since $x\in \tfrac34J$, $\frac{p_1}{q}\in J$. By Taylor's formula,
\begin{equation} \label{theo7_Taylor}
f(\tfrac{p_1}{q})=f(x)+f'(x)(\tfrac{p_1}{q}-x)+\tfrac12f''(\tilde{x}) (\tfrac{p_1}{q}-x)^2
\end{equation}
for some $\tilde{x}$ between $x$ and $p_1/q$. Thus  $\tilde{x} \in J$. Using (\ref{def_f}) and (\ref{theo7_Taylor}) we transform the
first inequality of~(\ref{vb+20}) into
\begin{equation} \label{theo7_Taylor_consequence_1}
|p_2-qf(\tfrac{p_1}{q})+\tfrac{q}{2}f''(\tilde{x})(x-\tfrac{p_1}{q})^2|\leq c_0\delta
\end{equation}
(see \cite[p.391]{Beresnevich-Dickinson-Velani-07:MR2373145} for details). Thus
$$
|qf(\tfrac{p_1}{q})-p_2| \leq |p_2-qf(\tfrac{p_1}{q})+\tfrac{q}{2}f''(\tilde{x})(x-\tfrac{p_1}{q})^2)|+ |\tfrac{q}{2}f''(\tilde{x})(x-\tfrac{p_1}{q})^2)|.
$$
In view of~(\ref{cond_g2p}), (\ref{vbvb}) and~(\ref{theo7_Taylor_consequence_1}),
$|qf(\tfrac{p_1}{q})-p_2| \leq  c_0\delta +\tfrac{Q}{2}c_2(\tfrac{C_1}{Q^2\delta})^2.$
By (\ref{e:003}), this further transforms into
$|qf(\tfrac{p_1}{q})-p_2| \leq 2 c_0\delta$. This inequality and (\ref{smallqlarge}) imply
$|f(\tfrac{p_1}{q})-\tfrac{p_2}{q}| \leq \tfrac{\delta }{Q}$.
Thus, we have shown that $(q,p_1,p_2)\in \cR^{c_0}_f(Q,\delta,J) $.
Therefore, in view of (\ref{vbvb}) we have that
$$
G\subset \Delta^{c_0}_f(Q,\delta,J,\rho)=\bigcup_{(q,p_1,p_2)\in\cR^{c_0}_f(Q,\delta,J)}
\big\{x:|x-p_1/q|\le \rho\big\}\quad\text{when}\quad\rho=\frac{C_1}{\delta Q^2}.
$$
By (\ref{theo7_Bisgrand}), this shows (\ref{vb+x}) and completes the proof.

{\small



\def\cprime{$'$} \def\cprime{$'$} \def\cprime{$'$} \def\cprime{$'$}
  \def\cprime{$'$} \def\cprime{$'$} \def\cprime{$'$} \def\cprime{$'$}
  \def\cprime{$'$} \def\cprime{$'$} \def\cprime{$'$} \def\cprime{$'$}
  \def\cprime{$'$} \def\cprime{$'$} \def\cprime{$'$} \def\cprime{$'$}
  \def\polhk#1{\setbox0=\hbox{#1}{\ooalign{\hidewidth
  \lower1.5ex\hbox{`}\hidewidth\crcr\unhbox0}}} \def\cprime{$'$}
  \def\cprime{$'$}

\bigskip

\noindent{\footnotesize VB\,: }\begin{minipage}[t]{0.9\textwidth}
\footnotesize{\sc University of York, Heslington, York, YO10 5DD,
England}\\
{\it E-mail address}\,:~~ \verb|vb8@york.ac.uk|
\end{minipage}

\smallskip

\noindent{\footnotesize EZ\,: }\begin{minipage}[t]{0.9\textwidth}
\footnotesize{\sc Institut de math\'ematiques de Jussieu, Universite de Paris 6, 75013, Paris, France}\\
{\it E-mail address}\,:~~ \verb|zorin@math.jussieu.fr|\quad\emph{or}\quad\verb|EvgeniyZorin@yandex.ru|
\end{minipage}

}

\end{document}